\documentclass[11pt,a4paper]{amsart}
\usepackage[T1]{fontenc}

\usepackage[initials]{amsrefs}
\usepackage{amssymb}
\usepackage{amsthm}
\usepackage{url}
\usepackage[all]{xy}

\newcommand{\trid}{\triangleright}
\newcommand{\toba}{\mathfrak{B}}

\newcommand{\GL}{\mathbf{GL}}
\newcommand{\PGL}{\mathbf{PGL}}
\newcommand{\Sym}{\mathbb{S}}

\newcommand{\img}{\mathrm{im}\,}
\newcommand{\sgn}{\mathrm{sgn}}
\newcommand{\id}{\mathrm{id}}
\newcommand{\N}{{\mathbb N}}
\newcommand{\Z}{{\mathbb Z}}
\newcommand{\C}{{\mathbb C}}

\newcommand{\B}{{\mathbb B}}

\theoremstyle{plain}
\newtheorem{thm}{Theorem}[section]
\newtheorem{lem}[thm]{Lemma}
\newtheorem{rem}[thm]{Remark}
\newtheorem{pro}[thm]{Proposition}

\newtheorem{exa}[thm]{Example}
\newtheorem{defn}[thm]{Definition}

\begin{document}
\title[Nichols algebras and twisting]{Nichols algebras associated to the transpositions of the symmetric group are twist-equivalent}
\author{L. Vendramin}

\thanks{This work was partially supported by CONICET}

\address{Departamento de Matem\'atica -- FCEyN,
Universidad de Buenos Aires, Pab. I -- Ciudad Universitaria (1428)
Buenos Aires -- Argentina}

\address{Instituto de Ciencias, Universidad de Gral. Sarmiento, J.M. Gutierrez
1150, Los Polvorines (1653), Buenos Aires -- Argentina}
\email{lvendramin@dm.uba.ar}


\begin{abstract}
	Using the theory of covering groups of Schur we prove that the two
	Nichols algebras associated to the conjugacy class of 
	transpositions in $\Sym_n$ are equivalent by twist and hence they 
	have the same Hilbert series. These algebras
	appear in the classification of pointed Hopf algebras and in the study of
	quantum cohomology ring of flag manifolds.
\end{abstract}

\maketitle

\section{Introduction}

Nichols algebras play a fundamental role in the classification of
finite-dimensional pointed Hopf algebras over $\C$.  They are graded Hopf
algebras 
\[
\toba(V)=\bigoplus_{n \geq 0}\toba^n(V)=\C\oplus V\oplus\toba^2(V)\oplus\cdots
\]
in the category of Yetter-Drinfeld modules over a Hopf algebra $H$, and they are
uniquely determined by $V$, the homogeneous component of $\toba(V)$ of degree one.

Let $H$ be the group algebra of a finite group $G$.  In the study of Nichols
algebras a basic question is to describe those Yetter-Drinfeld modules $V$ over
$H$ for which $\toba(V)$ is finite-dimensional. Whereas deep results were found
for the case where $G$ is abelian, \cites{MR2207786,MR2462836,MR2630042}, the 
situation is widely unknown for non-abelian groups $G$. 


The first examples of finite-dimensional pointed Hopf algebras with non-abelian
coradical appeared in \cite{MR1800714}, as bosonizations of Nichols
algebras related to the transpositions in $\Sym_3$ and $\Sym_4$.  The analogous
Nichols algebra over $\Sym_5$ was computed by Gra\~na, see \cite{zoo}.  These
Nichols algebras are computed from the conjugacy class of transpositions and a
$2$-cocycle (cocycle for short)
associated to this conjugacy class. The cocycles arise from a cohomology theory
defined for racks (see for example \cites{MR1800709,MR1990571,MR1994219}).
In \cite[Theorem 1.1]{ampa} it is proved that for all $n\in\mathbb{N}$, $n\geq4$, 
there are precisely two rack $2$-cocycles 
associated to the conjugacy class of transpositions in $\Sym_n$ that might 
have finite-dimensional Nichols algebras. 
Explicitly, one of these cocycles is the constant cocycle $-1$. The other one
is the cocycle given by 
\begin{equation}
    \label{eq:FK_cocycle}
\chi(\sigma,\tau)=\begin{cases}
	1 & \text{ if }\sigma(i)<\sigma(j),\\
	-1 & \text{ if }\sigma(i)>\sigma(j)\\
\end{cases}
\end{equation}
for all transpositions $\sigma$ and $\tau=(i\;j)$ with $i<j$.
For all $n\in\{4,5\}$ the Nichols algebra associated to the
conjugacy class of transpositions in $\Sym_n$ and any of the two cocycles $-1$, $\chi$ is 
finite-dimensional.  Moreover, both of these algebras have the
same Hilbert series. It is not known whether these algebras are finite-dimensional
for $n>5$.  The main result of this work is to connect these two algebras by
\emph{twisting the cocycle}. More precisely, we prove that the constant cocycle
$-1$ and $\chi$ are equivalent by twist. This gives an affirmative answer to a
question due to Andruskiewitsch,  see  \cite[Question 7]{AFGV}. 
However, the problem arose already earlier in the literature. For example, 
in the last paragraph of \cite{MR2106930}, Majid discusses the relationship 
between these two algebras and the related quadratic algebras.
To reach our main result, we use the existence of projective
representations of $\Sym_n$. Projective representations of $\Sym_n$ were
originally studied by Schur in 1911, see \cite{MR1820589} for an English
translation of his fundamental paper about this subject.
As a corollary of our result we obtain that for all
$n\geq4$ both Nichols algebras associated to the conjugacy class of transpositions of $\Sym_n$ 
have the same Hilbert series.

We recall briefly another application for Nichols algebras which may have
connections with the main result of this work. In \cite{MR0051508}, Borel
identified the cohomology ring of a flag manifold with $S_W$, the algebra of
coinvariants of the associated Coxeter group $W$. This admits certain
divided-difference operators which create classes of Schubert manifolds.  In
\cite{MR1667680}, Fomin and Kirillov introduced a new model for the Schubert
calculus of a flag manifold, realizing $S_W$ as a commutative subalgebra of a
noncommutative quadratic algebra $\mathcal{E}_W$, when $W$ is a symmetric
group. 
In \cite{MR2209265}, Bazlov proved that Nichols algebras provide the correct setting
for this model for Schubert calculus on a flag manifold.
It is an open problem whether the Nichols algebra associated to $\chi$ coincides 
with the quadratic algebra $\mathcal{E}_W$ \cites{MR1800714, MR2106930}.

\section{Preliminaries}
\subsection{Racks and cohomology}

We briefly recall basic facts about racks, see \cite{MR1994219} for more
information and references.  

\medskip
A \textit{rack} is a pair $(X,\triangleright)$
where $X$ is a non-empty set and $\triangleright:X\times X\to X$ is a function,
such that the map $x\mapsto i\triangleright x$ is bijective for all $i\in X$, 
and $i\triangleright(j\triangleright k)=(i\triangleright
j)\triangleright(i\triangleright k)$ for all $i,j,k\in X$.
A \textit{subrack} of $X$ is a non-empty subset $Y\subseteq X$ such
that $(Y,\triangleright)$ is also a rack.

\begin{exa}
A group $G$ is a rack with $x\triangleright y=xyx^{-1}$ for all $x,y\in G$.  If a subset
$X\subseteq G$ is stable under conjugation by $G$, then it is a subrack of $G$.
In particular, the conjugacy class of transpositions in $\Sym_n$ is a rack; it
will be denoted by $X_n$. 
\end{exa}

In this work we are interested only in racks which can be realized as a finite
conjugacy class of a group.  Let $X$ be such a rack. A map $q:X\times
X\to\C^\times$ is a \emph{$2$-cocycle} if and only if
\[
q_{x,y\trid z}q_{y,z}=q_{x\trid y,x\trid z}q_{x,z}
\]
for all $x,y,z\in X$. We write $Z_R^2(X, \C^\times)$ for
the set of all rack $2$-cocycles. Let $q,q'\in Z_R^2(X,\C^\times)$. We write
$q\sim q'$ if there exists $\gamma:X\to\C^\times$ such that 
\[
q_{x,y}=\gamma_{x\trid y}^{-1}q'_{x,y}\gamma_{y}
\]
for all $x,y\in X$. Since $\sim$ is an equivalence relation 
and
$Z_R^2(X,\C^\times)$ is stable under $\sim$ 
it is possible to define the \emph{second rack cohomology group} as
\[
H_R^2(X,\C^\times)=Z_R^2(X,\C^\times)/\sim.
\]
All these notions are based on the abelian cohomology theory of racks proposed
independently in \cite{MR1990571}, \cite{MR1800709}. For more details about
cohomology theories of racks see \cite[\S4]{MR1994219}.

\subsection{Nichols algebras}

We refer to \cite{MR1913436} for an introduction to Yetter-Drinfeld modules and
Nichols algebras.  

\medskip
Let $n\in\N$. We recall the well-known presentation
of the braid group $\B_n$ by generators and
relations. The group $\B_n$ has generators
$\sigma_1,\dots,\sigma_{n-1}$ and relations
\begin{align*}
	\sigma_i\sigma_j\sigma_i = \sigma_j\sigma_i\sigma_j &\quad\text{for $\vert i-j\vert=1$};\\
	\sigma_i\sigma_j = \sigma_j\sigma_i &\quad\text{for $\vert i-j\vert >1$}.
\end{align*}
There exists a canonical projection $\B_n\to\Sym_n$ that admits the so-called
Matsumoto section $\mu:\Sym_n\to\B_n$ such that $\mu\left( (i\;i+1)\right)=\sigma_i$. 
This section satisfies the following: $\mu(xy)=\mu(x)\mu(y)$
for any $x,y\in\Sym_n$ such that $l(xy)=l(x)+l(y)$, where $l$ is the length
function on $\Sym_n$.

\medskip
Let $V$ be a vector space over $\C$ and let $c\in\GL(V\otimes V)$. The map $c$
is a solution of the \emph{braid equation} if and only if 
\[
(c\otimes\id)(\id\otimes c)(c\otimes\id)=(\id\otimes c)(c\otimes\id)(\id\otimes c).
\]
If $c$ is a solution of the braid equation, we say that $(V,c)$ is a
\emph{braided vector space}.  A solution of the braid equation 
induces a representation $$\rho_n:\B_n\to\GL(V^{\otimes n})$$ 
defined by 
\[
\rho_n(\sigma_i)=\id^{\otimes(i-1)}\otimes c\otimes\id^{\otimes(n-i-1)}
\]
for $1\leq i\leq n-1$. Let 
\[
Q_n=\sum_{\sigma\in\Sym_n}\rho_n(\mu(\sigma)):V^{\otimes n}\to V^{\otimes n}
\]
for $n\geq2$, $Q_0=\id_\C$ and $Q_1=\id_V$.
The \emph{Nichols algebra} associated to the braided vector space $V$ is 
\[
\toba(V)=T(V)/\oplus_{n\geq2}\ker Q_n\simeq\oplus_{n\geq0}\img Q_n.
\]
By \cite[Theorem 4.14]{MR1994219}, Yetter-Drinfeld modules over group algebras 
can also be studied in terms of
racks and rack 2-cocycles. Therefore we are interested in Nichols algebras of 
braided vector spaces arising from racks and $2$-cocycles. 

\medskip
Let $(X,\trid)$ be a rack and let $q\in Z_R^2(X,\C^\times)$. We consider $V=\C
X$, the vector space with basis $x\in X$, and define $c:V\otimes V\to V\otimes
V$ as 
\[
c(x\otimes y)=q_{x,y} x\trid y\otimes x.
\]
Then $c$ is a solution of the braid equation. The Nichols algebra associated to
the pair $(X,q)$ is the Nichols algebra of the braided vector space $(V,c)$. 
This algebra will be denoted by $\toba(X,q)$.

\medskip
Recall that $X_n$ is defined as the rack associated to the conjugacy class of
transpositions in $\Sym_n$.  
In \cite[Theorem 1.1]{ampa} it is proved that 
there are two rack $2$-cocycles associated 
to $X_n$ that might have a finite-dimensional Nichols algebra. 
One is the constant $2$-cocycle $-1$.
The other is the $2$-cocycle $\chi$ given by Equation \eqref{eq:FK_cocycle}.

\begin{rem}
It can be checked directly that the $2$-cocycles $-1$ and $\chi$ associated to
the rack $X_3$ are cohomologous. Then the Nichols algebras $\toba(X_3,\chi)$
and $\toba(X_3,-1)$ are isomorphic and hence they have the same Hilbert series.
\end{rem}

For all $m\in\N$ let $(m)_t=1+t+\cdots+t^{m-1}\in\Z[t]$.

\begin{exa}
	The Nichols algebras $\toba(X_4,-1)$ and $\toba(X_4,\chi)$ 
	have both dimension $576$. In both cases the Hilbert series is 
	\[
	\mathcal{H}_4(t)=(2)^2_t (3)^2_t (4)^2_t.
	\]
	These algebras appeared first in
	\cites{MR1667680,MR1800714}. For more information about these 
	algebras see \cite[Theorem 2.4 and Proposition 2.5]{GG} or 
	\cite[Proposition 5.11]{GHV}.
\end{exa}

\begin{exa}
	The Nichols algebras $\toba(X_5,-1)$ and $\toba(X_5,\chi)$ have both
	dimension $8294400$. In both cases the Hilbert series is 
	\[
	\mathcal{H}_5(t)=(4)^4_t (5)^2_t (6)^4_t.
	\]
	These algebras were first computed by
	Gra\~na, \cite{zoo}. For more information about these algebras 
	see \cite[Theorem 2.4 and Proposition 2.5]{GG} or \cite[Proposition 5.13]{GHV}.
\end{exa}

\subsection{Twisting}

In \cite[Section 3.4]{AFGV} it is shown how to relate two rack $2$-cocycles by
a twisting in such a way that some properties of the corresponding Nichols 
algebras are preserved. 
This method is based on the twisting method of \cite{MR1298746} and its
relationship with the bosonization given in \cite{MR1711340}.  

\medskip
Let $X$ be a subrack of a conjugacy class of a group $G$. Let $q$ be a rack
$2$-cocycle on $X$ and let $\phi$ be a group $2$-cocycle on $G$.  Define
$q^\phi: X\times X\to \C^\times$ by
\begin{equation}
\label{eqn:twisting}
q^\phi_{x,y} = \phi(x, y)\phi(x\trid y, x)^{-1}\,q_{x,y}
\end{equation}
for $x,y\in X$.

\begin{rem}
Let $X$ be a rack and $q\in H_R^2(X,\C^\times)$.  For a map $\phi: X\times X\to
\C^\times$ define $q^\phi$ by Equation \eqref{eqn:twisting}.  Then $q^\phi$
is a rack $2$-cocycle if and only if 
\begin{multline}
\label{eqn:condition}
\phi(x,z)\phi(x\trid y, x\trid z)\phi(x\trid(y\trid z),x)\phi(
y\trid z, y)
\\= \phi(y,z)\phi(x, y\trid z)\phi(x\trid(y\trid z), x\trid y)\phi(x\trid z, x)
\end{multline}
for any $x,y,z\in X$. Thus, if $X$ is a subrack of a group $G$ and
$\phi$ is a group 2-cocycle, $\phi\in Z^2(G, \C^{\times})$, then  
$\phi\vert_{X\times X}$ 
satisfies Equation \eqref{eqn:condition}. 
See \cite[Remark 3.10]{AFGV} or \cite[Theorem 7.1]{MR1990571} for a similar result. 
\end{rem}

\begin{lem}
\label{lem:series}
The Hilbert series of $\toba(X, q)$ and $\toba(X, q^\phi)$ are equal.
\end{lem}
\begin{proof}
See \cite[Section 3.4]{AFGV}.
\end{proof}

\begin{defn}
\label{defn:twisting} 
The $2$-cocycles $q$ and $q'$ on $X$ are \emph{equivalent by twist} if there
exists $\phi: X\times X\to \C^\times$ such that $q' = q^\phi$ as in
\eqref{eqn:twisting}.
\end{defn}

\section{The Schur cover of $\Sym_n$}

\subsection{Projective representations and covering groups}

We review some aspects of Schur's theory of projective
representations and construct the Schur cover of $\Sym_n$. See \cites{MR1486039,
MR1200015, MR1820589} for details. 

\medskip
A \emph{projective representation} of a finite group $G$ is a group
homomorphism $G\to\mathbf{PGL}(V)$. Equivalently, such a representation may be
viewed as a map $f:G\to\mathbf{GL}(V)$ such that 
\[
f(x)f(y)=\phi(x,y)f(xy)
\]
for all $x,y\in G$ and suitable scalars $\phi(x,y)\in\C^\times$. The map $G\times G\to\C^\times$, 
$(x,y)\mapsto \phi(x,y)$, is called a \emph{factor set}. The associativity of
the group $\GL(V)$ implies the \emph{$2$-cocycle condition} of
the factor set $\phi$:
\begin{equation}
\phi(x,y)\phi(xy,z)=\phi(x,yz)\phi(y,z)
\end{equation}
for all $x,y,z\in G$.

Two projective representations $\rho_1:G\to\GL(V_1)$ and $\rho_2:G\to\GL(V_2)$ are
\emph{equivalent} if there exists a $\C$-linear isomorphism  
$S:V_1\to V_2$ and a map $b:G\to\C^\times$ 
such that
\[
b(x)S\rho_1(x)S^{-1}=\rho_2(x)
\]
for all $x\in G$. Two factor sets $\phi$ and $\phi'$ are \emph{equivalent} if
they differ only by a factor $b_xb_y/b_{xy}$ for some $b:G\to\C^\times$. The
\emph{Schur multiplier} of $G$ is the abelian group of factor sets modulo
equivalence. It is isomorphic to the second cohomology group
$H^2(G,\C^\times)$.

Recall that a central extension of $G$ is a pair $(E,p)$, where $p:E\to G$ is
a surjective group homomorphism such that $\ker p$ is contained in the center
of the group $E$. Schur proved that every finite group $G$ has a central
extension $(E,p)$ with the property that every projective representation $\rho$
of $G$ lifts to an ordinary representation $\bar\rho$ of $E$ such that the
diagram
\begin{equation*}
\xymatrix{E\ar@{->}^{\bar\rho}[0,2]\ar@{->}[1,0]_{p}&
&\GL(V)\ar@{->}[1,0]^{} \\ G\ar@{->}^{\rho}[0,2]& & \PGL(V).}
\end{equation*}
commutes.

There exist extensions with $\ker p\simeq H^2(G,\C^\times)$. Moreover,
$H^2(G,\C^\times)$ is the unique minimal possibility for $\ker p$. These
minimal central extensions of $G$ are called \emph{Schur covering groups} of
$G$.

\medskip
We recall that $\Sym_n$ has a Coxeter group presentation
\begin{multline*}
\Sym_n = \langle s_1,\cdots,s_{n-1}\mid s_i^2=(s_js_{j+1})^3=(s_ks_l)^2=1\rangle\\
(1\leq i\leq n-1,\,1\leq j\leq n-2,\,k\leq l-2),
\end{multline*}
where $s_1,\cdots,s_{n-1}$ denote the adjacent transpositions of $\Sym_n$ defined by
\[
s_1=(1\,2),\, s_2=(2\,3),\dots,s_{n-1}=(n-1\,n).
\]

\begin{thm}
\label{thm:schur}
Given $n\geq4$, define the group $T_n$ as follows
\begin{multline*}
T_n = \langle z,t_1,\cdots,t_{n-1}\mid t_i^2=(t_jt_{j+1})^3=1,\;(t_kt_l)^2=z,\;z^2=[z,t_i]=1\rangle\\
(1\leq i\leq n-1,\,1\leq j\leq n-2,\,k\leq l-2).
\end{multline*}
Then $T_n$ is a Schur covering group of $\Sym_n$. Therefore, there
exists a central extension
	\[
	0\to A\xrightarrow{i} T_n\xrightarrow{p}\Sym_n\to1,
	\]
where $A=\langle z\rangle$.
\end{thm}

\begin{proof}
\cite[Theorem 2.12.3]{MR1200015}.
\end{proof}

\subsection{Transpositions in $T_n$}

We want to study the lifting of the conjugacy class of transpositions in
$\Sym_n$.  See \cite[Section 3]{MR2562037} for a detailed description about
permutations in $T_n$. 

\begin{rem}
\label{rem:conjugation}
Let $t\in T_n$. For any $\sigma\in\Sym_n$ we have that
$p^{-1}(\sigma)=\{\bar\sigma,\bar\sigma z\}$.  Since the involution $z$ is a central
element of $T_n$, the group $\Sym_n$ acts on $T_n$ by conjugation:
\[
\sigma\trid t=\bar\sigma t(\bar\sigma)^{-1}=(\bar\sigma z)t(\bar\sigma z)^{-1}.
\]
Therefore it is possible to write the conjugation in $T_n$ as $\sigma\trid
t=\sigma t\sigma^{-1}$, where $t\in T_n$ and $\sigma\in\Sym_n$. 
\end{rem}

\begin{defn}
\label{defn:permutations}
For $i,j\in\N$ such that $1\leq i,j\leq n$, $i\ne j$, let $[i\;j]$ be an element of $T_n$
defined inductively as
\begin{align*}
[i\;i+1]&=t_i,\\
[i\;j]&=t_i\triangleright [i+1\;j]z,\quad\text{for $i+1<j$},\\
[j\;i]&=[i\;j]z.
\end{align*}
\end{defn}

\begin{lem} 
	\label{lem:general} 
	Let $i,j,k\in\{1,...,n\}$. Then $s_k\trid [i\;j]=[s_k(i)\;s_k(j)]z$.
\end{lem}

\begin{proof}
Multiplying both sides by $z$ if needed, we may assume that $i<j$. If
$\{k,k+1\}\cap\{i,j\}=\emptyset$ then the claim follows from \cite[Paragraph 6,
III]{MR1820589}.  If $k=i-1$ then the claim follows from Definition
\ref{defn:permutations}. The case $k=i$ follows from the case $k=i-1$ by
applying $s_{i-1}$. Since $s_j\trid t_{j-1}=s_{j-1}\trid t_j$, a
straightforward computation settles the case $k=j$. Finally, the case $k=j-1$
follows from the case $k=j$ by applying $s_j$.
\end{proof}


\begin{pro}
\label{pro:conjugation}
Let $l\in\N$, $\sigma=s_{i_1}s_{i_2}\cdots s_{i_l}\in\Sym_n$ and $i,j\in\{1,...,n\}$. 
Then
\[
\sigma\trid[i\;j]=[\sigma(i)\;\sigma(j)]z^l.
\]
\end{pro}

\begin{proof}
Follows from Lemma \ref{lem:general} by induction on $l$.
\end{proof}


\subsection{Nichols algebras over symmetric groups}
Recall that $X_n$ is the rack of transpositions in $\Sym_n$. There exist two
rack $2$-cocycles that we want to consider.  One of these rack $2$-cocycles is
the constant cocycle $-1$. The other one is the $2$-cocycle given by Equation 
\eqref{eq:FK_cocycle}.
%


\begin{lem}
\label{lem:repr}
Let $A$ be an abelian group and let $\psi:A\to\C^\times$ be a group homomorphism. For $\phi\in
Z^2(\mathbb{S}_n,A)$ define $\phi_\psi(x,y)=\psi(\phi(x,y))$. Then
$\phi_\psi\in Z^2(\Sym_n,\C^\times)$.
\end{lem}
\begin{proof}
Follows from \cite[Chapter 6, $\S$3]{MR1486039}.
\end{proof}

\begin{lem}
\label{lem:section}
There exists a section $s:\Sym_n\to T_n$ such that if $\tau=(i\;j)$, $i<j$, 
then 
\begin{equation}
	s(\sigma)\trid s(\tau)=
	\begin{cases}
		s(\sigma\trid\tau)z & \text{if $\sigma(i)<\sigma(j)$},\\
		s(\sigma\trid\tau) & \text{if $\sigma(i)>\sigma(j)$}
	\end{cases}
\end{equation}
for all $\sigma$.
\end{lem}

\begin{proof}
	By Theorem \ref{thm:schur} there exists a central extension
	\[
	0\to A\xrightarrow{i} T_n\xrightarrow{p}\Sym_n\to1,
	\]
	where $A=\langle z\rangle$. Take any set-theoretical section $\bar
	s:\Sym_n\to T_n$ such that $\bar s(\id)=1$ and define a map 
	$s:\Sym_n\to T_n$ by 
	\begin{equation}
		s(\pi)=\begin{cases}
			\bar s(\pi) & \text{if $\pi\notin X_n$},\\
			[i\;j] & \text{if $\pi=(i\;j)\in X_n$, with $i<j$}.
		\end{cases}
	\end{equation}
	Then $ps=\id$ and $s(\id)=1$. 
	Since $\sigma\in X_n$, the length of $\sigma$ is $1$.  Remark
	\ref{rem:conjugation} and Proposition \ref{pro:conjugation} imply that
	\[
	s(\sigma)\trid s(\tau)=\sigma\trid s(\tau)=\sigma[i\;j]\sigma^{-1}=[\sigma(i)\;\sigma(j)]z.
	\]
	Hence the claim follows.
\end{proof}

\begin{thm}
Let $n\geq4$. The rack $2$-cocycles $\chi$ and $-1$ associated to $X_n$ are twist equivalent.
Hence the Hilbert series of $\toba(X_n, -1)$ and $\toba(X_n, \chi)$ are equal.
\end{thm}

\begin{proof}
By Theorem \ref{thm:schur} there exists a central extension
\[
0\to A\xrightarrow{i} T_n\xrightarrow{p}\Sym_n\to1,
\]
where $A=\langle z\rangle$. Let $s:\Sym_n\to T_n$ be the section of Lemma
\ref{lem:section} and let $\phi(x,y)\in A$ be defined by the equation
\[
s(x)s(y)=i(\phi(x,y))s(xy).
\]
Then $\phi\in Z^2(\Sym_n,A)$. For any group homomorphism $\psi:A\to\C^\times$ 
apply Lemma \ref{lem:repr} to get a group $2$-cocycle $\phi_\psi\in Z^2(\Sym_n,\C^\times)$.
Take $\psi=\sgn$, where $\psi(z)=-1$. We claim that 
\[
-1 = \phi_\psi(\sigma,\tau)\phi_\psi(\sigma\trid\tau,\sigma)^{-1}\chi(\sigma,\tau)
\]
for all $\sigma,\tau\in X_n$. In fact,
\begin{align*}
	\phi_\psi(\sigma,\tau)&\phi_\psi(\sigma\trid\tau,\sigma)^{-1}\\
	&= \psi\left( s(\sigma)s(\tau)s(\sigma\tau)^{-1}\right)\psi\left( s(\sigma\tau\sigma^{-1})s(\sigma)s(\sigma\tau)^{-1})\right)^{-1}\\
	& = \psi\left( s(\sigma)s(\tau)s(\sigma\tau)^{-1}(s(\sigma\tau\sigma^{-1})s(\sigma)s(\sigma\tau)^{-1})^{-1}\right)\\
	& = \psi\left( s(\sigma)s(\tau)s(\sigma\tau)^{-1}s(\sigma\tau)s(\sigma)^{-1}s(\sigma\tau\sigma^{-1})^{-1}  \right)\\
	& = \psi\left( s(\sigma)s(\tau)s(\sigma)^{-1}s(\sigma\tau\sigma^{-1})^{-1}\right) 
\end{align*}
Therefore the first claim follows from Lemma \ref{lem:section}. For the second claim use 
Lemma \ref{lem:series}.
\end{proof}

\subsection*{Acknowledgement}
I am grateful to Mat\'\i as Gra\~na for valuable comments and discussions and 
to Nicol\'as Andruskiewitsch for useful discussions and for suggesting me this
project. I thank Istv\'an Heckenberger for his careful reading of this work. 
I also thank the referee for suggesting Lemma \ref{lem:general} which simplified 
the original presentation of the paper.

\section*{References}


\end{document}